\documentclass[12pt,reqno]{amsart}
\usepackage{amssymb}

\newtheorem{theorem}{Theorem}[section]
\newtheorem{lemma}[theorem]{Lemma}
\newtheorem{proposition}[theorem]{Proposition}

\theoremstyle{definition}

\numberwithin{equation}{section}

\begin{document}

\baselineskip=15.5pt

\title[Equivariant bundles and logarithmic connections on toric varieties]{Equivariant
principal bundles and logarithmic connections on toric varieties}

\author[I. Biswas]{Indranil Biswas}

\address{School of Mathematics, Tata Institute of Fundamental
Research, Homi Bhabha Road, Mumbai 400005, India}

\email{indranil@math.tifr.res.in}

\author[A. Dey]{Arijit Dey}

\address{Department of Mathematics, Indian Institute of Technology-Madras, Chennai, India}

\email{arijitdey@gmail.com}

\author[M. Poddar]{Mainak Poddar}

\address{Departamento de
Matem\'aticas, Universidad de los Andes, Bogota, Colombia}

\email{mainakp@gmail.com}

\subjclass[2010]{14M27, 14L30, 14M17}

\keywords{Smooth toric variety, logarithmic connection, equivariant principal bundle.}

\thanks{The first-named author is supported by a J. C. Bose Fellowship.}

\begin{abstract}
Let $M$ be a smooth complex projective toric variety equipped with an action of
a torus $T$, such that the complement $D$ of the open $T$--orbit in $M$ is a simple
normal crossing divisor. Let $G$ be a complex reductive affine algebraic group. We
prove that an algebraic principal $G$--bundle $E_G\,\longrightarrow\, M$ admits a
$T$--equivariant structure if and only if $E_G$ admits a logarithmic connection
singular over $D$. If $E_H\,\longrightarrow\, M$ is a $T$--equivariant algebraic
principal $H$--bundle, where $H$ is any complex affine algebraic group, then $E_H$
in fact has a canonical integrable logarithmic connection singular over $D$.
\end{abstract}

\maketitle

\section{Introduction}

Our aim here is to give characterizations of the equivariant principal bundles on 
smooth complex projective toric varieties.

Let $M$ be a smooth complex projective toric variety equipped with an action of
a torus $T$
$$
\rho\, :\, T\times M\, \longrightarrow\, M\, .
$$
For any point $t\, \in\, T$, define the automorphism
$$
\rho_t\, :\, M\,\longrightarrow\, M\, , \ ~ x\, \longmapsto\, \rho(t\, , x)\, .
$$
We assume that the complement $D$ of the open $T$--orbit in $M$ is a simple
normal crossing divisor.

Let $G$ be a complex reductive affine algebraic group, and let $E_G$ be an algebraic
principal $G$--bundle on $M$. In Proposition \ref{prop3} we prove the following:

\textit{The principal $G$--bundle $E_G$ admits a $T$--equivariant
structure if and only if the pulled back principal $G$--bundle $\rho^*_t E_G$
is isomorphic to $E_G$ for every $t\, \in\, T$.}

When $G\,=\, \text{GL}(n, {\mathbb C})$, this result was proved by Klyachko
\cite[p. 342, Proposition 1.2.1]{Kl}.

Using the above characterization of $T$--equivariant principal $G$--bundles on $M$,
we prove the following (see Theorem \ref{thm1}):

\textit{The principal $G$--bundle $E_G$ admits a logarithmic connection singular
over $D$ if and only if $E_G$ admits a $T$--equivariant structure.}

The ``if'' part of Theorem \ref{thm1} does not require $G$ to be reductive. More
precisely,
any $T$--equivariant principal $H$--bundle $E_H\,\longrightarrow\, M$, where $H$ is any
complex affine algebraic group, admits a canonical integrable logarithmic connection
singular over $D$ (see Proposition \ref{prop2}).

\section{Equivariant bundles}\label{se2}

Let ${\mathbb G}_m\,=\, {\mathbb C}\setminus\{0\}$ be the multiplicative group. Take a
complex algebraic group $T$ which is isomorphic to a product of copies of ${\mathbb G}_m$.
Let $M$ be a smooth irreducible
complex projective variety equipped with an algebraic action of $T$
\begin{equation}\label{e1}
\rho\, :\, T\times M\, \longrightarrow\, M
\end{equation}
such that
\begin{itemize}
\item there is a Zariski open dense subset $M^0\, \subset\, M$ with $\rho(T\, ,
M^0) \,=\, M^0$,

\item the action of $T$ on $M^0$ is free and transitive, and

\item the complement $M\setminus M^0$ is a simple normal crossing divisor of $M$.
\end{itemize}
In particular, $M$ is a smooth projective toric variety. Note that $M^0$ is the unique
$T$--orbit in $M$ with trivial isotropy.

Let $G$ be a connected complex affine algebraic group. A $T$--\textit{equivariant}
principal $G$--bundle on $M$ is a pair $(E_G\, , \widetilde{\rho})$, where
$$
p\, :\, E_G\, \longrightarrow\, M
$$
is an algebraic principal $G$--bundle, and
$$
\widetilde{\rho}\, :\, T\times E_G\,\longrightarrow\, E_G
$$
is an algebraic action of $T$ on the total space of $E_G$, such that
\begin{itemize}
\item $p\circ\widetilde{\rho}\,=\, \rho\circ ({\rm Id}_{_T}\times p)$, where
$\rho$ is the action in \eqref{e1}, and

\item the actions of $T$ and $G$ on $E_G$ commute.
\end{itemize}

Fix a point $x_0\, \in\, M^0\, \subset\, M$. Let
\begin{equation}\label{e2}
\iota\, :\, \rho(T\, ,x_0)\,=\, M^0\, \hookrightarrow\, M
\end{equation}
be the inclusion map. Let $M^0\times G$ be the trivial principal $G$--bundle on
$M^0$. It has a tautological integrable algebraic connection given by its trivialization.

Let $(E_G\, ,\widetilde{\rho})$ be a $T$--equivariant principal $G$--bundle on $M$.
Fix a point $z_0\, \in\, (E_G)_{x_0}$. Using $z_0$,
the action $\widetilde{\rho}$ produces an isomorphism of principal $G$--bundles
between $M^0\times G$ and the restriction $E_G\vert_{M^0}$. This isomorphism
of principal $G$--bundles is uniquely determined by the following two conditions:
\begin{itemize}
\item this isomorphism is $T$--equivariant (the action of $T$ on
$M^0\times G$ is given by the action of $T$ on $M^0$), and

\item it takes the point $z_0\,\in\, E_G$ to $(x_0\, ,e)\,\in\, M^0\times G$.
\end{itemize}
Using this trivialization of $E_G\vert_{M^0}$, the tautological integrable algebraic
connection on $M^0\times G$ produces an integrable algebraic connection ${\mathcal D}^0$
on $E_G\vert_{M^0}$. We note that this connection ${\mathcal D}^0$ is independent of the
choice of the points $x_0$ and $z_0$. Indeed, the flat sections for ${\mathcal D}^0$ are
precisely the orbits of $T$ in $E_G\vert_{M^0}$. Note that this description of
${\mathcal D}^0$ does not require choosing base points in $M^0$ and $E_G\vert_{M^0}$.

In Proposition \ref{prop2} it will be shown that ${\mathcal D}^0$ extends to a 
logarithmic connection on $E_G$ over $M$ singular over the simple normal crossing 
divisor $M\setminus M^0$.

\section{Logarithmic connections}

\subsection{A canonical trivialization}

The Lie algebra of $T$ will be denoted by $\mathfrak t$. Let
\begin{equation}\label{cV}
{\mathcal V}\, :=\, M\times {\mathfrak t}\,\longrightarrow\, M
\end{equation}
be the trivial vector bundle with fiber $\mathfrak t$. The holomorphic
tangent bundle of $M$ will be denoted by $TM$. Consider the action of $T$ on $M$
in \eqref{e1}. It produces a homomorphism of ${\mathcal O}_M$--coherent sheaves
\begin{equation}\label{beta}
\beta\, :\, {\mathcal V}\,\longrightarrow\, TM\, .
\end{equation}

Let
$$
D\,:=\, M\setminus M^0
$$
be the simple normal crossing divisor of $M$. Let
\begin{equation}\label{e4}
TM(-\log D)\, \subset\, TM
\end{equation}
be the corresponding logarithmic tangent bundle. We recall that $TM(-\log D)$ is characterized
as the maximal coherent subsheaf of $TM$ that preserves ${\mathcal O}_M(-D)\, \subset\,
{\mathcal O}_M$ for the derivation action of $TM$ on ${\mathcal O}_M$.

\begin{lemma}\label{prop1}
\mbox{}
\begin{enumerate}
\item The image of $\beta$ in \eqref{beta} is contained in the subsheaf $TM(-\log D)\,
\subset\, TM$.

\item The resulting homomorphism $\beta\, :\, {\mathcal V}\,\longrightarrow\,
TM(-\log D)$ is an isomorphism.
\end{enumerate}
\end{lemma}

\begin{proof}
The divisor $D$ is preserved by the action of $T$ on $M$. Therefore, the action of $T$ on
${\mathcal O}_M$, given by the action of $T$ on $M$, preserves the subsheaf ${\mathcal O}_M(-D)$. From
this it follows immediately that
the subsheaf ${\mathcal O}_M(-D)\, \subset\, {\mathcal O}_M$ is preserved by the derivation action of
the subsheaf $$\beta({\mathcal V})\, \subset\, TM\, .$$ Therefore, we conclude that
$\beta({\mathcal V})\, \subset\, TM(-\log D)$.

It is known that the vector bundle $TM(-\log D)$ is holomorphically trivial
\cite[p. 87, Proposition 2]{Fu}. We note that Proposition 2 of \cite[p. 87]{Fu} says that
$\Omega^1_M(\log D)$ is holomorphically trivial. But $\Omega^1_M(\log D)^*\,=\, TM(-\log D)$,
and hence $TM(-\log D)$ is also holomorphically trivial.

So, both ${\mathcal V}$ and $TM(-\log D)$ are trivial vector bundles, and $\beta$ is a
homomorphism between them which is an isomorphism over the open subset $M^0$. From this it
can be deduced that $\beta$ is an isomorphism over entire $M$. To see this, consider the homomorphism
$$
\bigwedge\nolimits^r \beta\, :\, \bigwedge\nolimits^r{\mathcal V}\,\longrightarrow\,
\bigwedge\nolimits^r TM(-\log D) 
$$
induced by $\beta$, where $r\,=\, \dim_{\mathbb C}T \,=\, \text{rank}({\mathcal V})$. So
$\bigwedge\nolimits^r \beta$ is a holomorphic section of the line bundle
$(\bigwedge^r TM(-\log D))\bigotimes
(\bigwedge^r{\mathcal V})^*$. This line bundle $(\bigwedge^r TM(-\log D))\bigotimes
(\bigwedge^r{\mathcal V})^*$ is holomorphically trivial because both ${\mathcal V}$ and
$TM(-\log D)$ are holomorphically trivial. Fixing a trivialization of
$(\bigwedge^r TM(-\log D))\bigotimes (\bigwedge^r{\mathcal V})^*$, we
consider $\bigwedge^r \beta$ as a holomorphic function on $M$. This function is
nowhere vanishing because it does not vanish on $M^0$ and holomorphic functions on
$M$ are constants. Since $\bigwedge^r \beta$ is nowhere vanishing, the
homomorphism $\beta$ is an isomorphism.
\end{proof}

\subsection{A canonical logarithmic connection on equivariant bundles}\label{se3.2}

The Lie algebra of $G$ will be denoted by $\mathfrak g$.

Let $p\, :\, E_G\, \longrightarrow\, M$ be an algebraic principal $G$--bundle. Consider
the differential
\begin{equation}\label{dp}
dp\, :\, TE_G\,\longrightarrow\, p^*TM\, ,
\end{equation}
where $TE_G$ is the algebraic tangent bundle of $E_G$. The kernel of $dp$ will be denoted
by $T_{E_G/M}$. Using the action of $G$ on $E_G$, this subbundle $T_{E_G/M}\, \subset\,
TE_G$ is identified with the trivial vector bundle over $E_G$ with fiber $\mathfrak g$.

The action of $G$ on $E_G$ produces an action of $G$ on $TE_G$.
So we get an action of $G$ on the quasicoherent sheaf $p_*TE_G$ on $M$. The invariant part
$$
\text{At}(E_G)\,:=\, (p_*TE_G)^G\, \subset\, p_*TE_G
$$
is a locally free coherent sheaf; its coherence property follows from the fact that the
action of $G$ on the fibers of $p$ is transitive, implying that a $G$--invariant section of
$(TE_G)\vert_{p^{-1}(x)}$, $x\, \in\, M$, is uniquely determined by its evaluation at
just one point of the fiber $p^{-1}(x)$. Also note that $\text{At}(E_G)\,=\,(TE_G)/G$.
This $\text{At}(E_G)$ is known as the \textit{Atiyah
bundle} for $E_G$. Since $T_{E_G/M}$ is identified with $E_G\times \mathfrak g$, the
invariant direct image
$(p_*T_{E_G/M})^G$ is identified with the adjoint vector bundle $$\text{ad}(E_G)\,:=\,
E_G\times^G\mathfrak g\,\longrightarrow\, M$$
associated to $E_G$ for the adjoint action of $G$ on $\mathfrak g$. We note that
$\text{ad}(E_G)\,=\, T_{E_G/M}/G$. Now the differential $dp$
in \eqref{dp} produces a short exact sequence of holomorphic vector bundles on $M$
\begin{equation}\label{e3}
0\,\longrightarrow\, \text{ad}(E_G)\,\longrightarrow\, \text{At}(E_G)\,
\stackrel{\phi}{\longrightarrow}\, TM \,\longrightarrow\, 0\, ,
\end{equation}
which is known as the Atiyah exact sequence. A holomorphic connection on $E_G$ over $M$
is a holomorphic splitting 
$$
TM\,\longrightarrow\, \text{At}(E_G)
$$
of \eqref{e3} \cite{At}.

As before, setting $D\,=\, M\setminus M^0$, define
$$
\text{At}(E_G)(-\log D) \,:=\, \phi^{-1}(TM(-\log D))\, \subset\, \text{At}(E_G)\, ,
$$
where $\phi$ is the projection in \eqref{e3} and $TM(-\log D)$ is the subsheaf in
\eqref{e4}. So \eqref{e3} gives the following short exact sequence of holomorphic
vector bundles on $M$
\begin{equation}\label{e5}
0\,\longrightarrow\, \text{ad}(E_G)\,\longrightarrow\, \text{At}(E_G)(-\log D)\,
\stackrel{\phi}{\longrightarrow}\, TM(-\log D) \,\longrightarrow\, 0\, .
\end{equation}

A \textit{logarithmic connection} on $E_G$, with singular locus $D$, is
a holomorphic homomorphism
$$
\delta\, :\, TM(-\log D)\,\longrightarrow\,\text{At}(E_G)(-\log D)
$$
such that $\phi\circ\delta$ is the identity automorphism of $TM(-\log D)$, where $\phi$
is the homomorphism in \eqref{e5}. Just like the curvature of a connection, the
curvature of a logarithmic connection $\delta$ on $E_G$ is the obstruction
for the homomorphism $\delta$ to preserve the Lie algebra structure of the sheaf of sections
of $TM(-\log D)$ and $\text{At}(E_G)(-\log D)$
given by the Lie bracket of vector fields. In particular, $\delta$ is called \textit{integrable} (or
\textit{flat}) if it preserves the Lie algebra structure of the sheaf of sections of
$TM(-\log D)$ and $\text{At}(E_G)(-\log D)$ given by the Lie bracket of vector fields.

\begin{proposition}\label{prop2}
Let $(E_G\, , \widetilde{\rho})$ be a $T$--equivariant principal $G$--bundle on $M$.
Then $E_G$ admits an integrable logarithmic connection that restricts to the connection
${\mathcal D}^0$ on $M^0$ constructed in Section \ref{se2}.
\end{proposition}

\begin{proof}
Let 
$$
\widetilde{\mathcal V}\, :=\, E_G\times {\mathfrak t}\,\longrightarrow\, E_G
$$
be the trivial vector bundle over $E_G$ with fiber $\mathfrak t$. Note that
$p^*{\mathcal V}\,=\, \widetilde{\mathcal V}$, where $\mathcal V$ is the vector
bundle in \eqref{cV}, and $p$, as before, is the projection of $E_G$ to $M$.

The action $\widetilde{\rho}$ of $T$ on $E_G$ produces a homomorphism
\begin{equation}\label{e6}
\widetilde{\beta}\, :\, \widetilde{\mathcal V}\,\longrightarrow\, TE_G\, .
\end{equation}
Since $p^{-1}(D)$ is preserved by the action of $T$ on $E_G$, the induced action of $T$
on ${\mathcal O}_{E_G}$ preserves the subsheaf ${\mathcal O}_{E_G}(- p^{-1}(D))$. 
Hence the image of $\widetilde{\beta}$ lies inside the subsheaf
$$
TE_G(-\log p^{-1}(D))\, \subset\, TE_G\, .
$$
Note that $p^{-1}(D)$ is a simple normal crossing divisor on $E_G$ because $D$ is
a simple normal crossing divisor on $M$.

In Lemma \ref{prop1}(2) we saw that $\beta$ is an isomorphism. Consider
$$
p^*\beta^{-1}\, :\, p^*(TM(-\log D))
\,\longrightarrow\, p^*{\mathcal V}\,=\, \widetilde{\mathcal V}\, .
$$
Pre-composing this with $\widetilde{\beta}$ in \eqref{e6}, we have
$$
\widetilde{\beta}\circ (p^*\beta^{-1})\, :\, p^*(TM(-\log D))
\,\longrightarrow\,TE_G(-\log p^{-1}(D))\, .
$$
We observe that the homomorphism $\widetilde{\beta}\circ (p^*\beta^{-1})$ is
$G$--equivariant for the trivial action of $G$ on $p^*(TM(-\log D))$ and the
action of $G$ on $TE_G(-\log p^{-1}(D))$ induced by the action of $G$ on $E_G$.
Therefore, taking the $G$--invariant parts of the direct images by $p$, the above
homomorphism $\widetilde{\beta}\circ (p^*\beta^{-1})$ produces a homomorphism
$$
\beta'\, :\,TM(-\log D)\,=\, (p_*p^*(TM(-\log D)))^G
$$
$$
\,\longrightarrow\, (p_*TE_G(-\log p^{-1}(D)))^G\,=\,
\text{At}(E_G)(-\log D)\, .
$$
It is now straightforward to check that the above homomorphism $\beta'$ produces a
holomorphic splitting of the exact sequence in \eqref{e5}. Therefore, $\beta'$ defines
a logarithmic connection on $E_G$ singular on $D$. The restriction of this logarithmic
connection to $M^0$ clearly coincides with the connection ${\mathcal D}^0$ constructed
in Section \ref{se2}.
\end{proof}

\section{A criterion for equivariance}

For each point $t\, \in\, T$, define the automorphism
$$
\rho_t\, :\, M\,\longrightarrow\, M\, , \ ~ x\, \longmapsto\, \rho(t\, , x)\, ,
$$
where $\rho$ is the action in \eqref{e1}. If $(E_G\, , \widetilde{\rho})$ is
a $T$--equivariant principal $G$--bundle on $M$, then clearly the map
$$
E_G\, \longrightarrow\, E_G\, , \ ~ z\, \longmapsto\,\widetilde{\rho}(t\, , z)
$$
is an isomorphism of the principal $G$--bundle $\rho^*_t E_G$ with $E_G$. The
aim in this section is to prove a converse of it.

Take an algebraic principal $G$--bundle
$$p\, :\, E_G\, \longrightarrow\, M\, .$$
Let $\mathcal G$ be the set of all pairs of the form $(t\, ,f)$, where $t\, \in\, T$ and
$$
f\, :\, E_G\, \longrightarrow\, E_G
$$
is an algebraic automorphism of the variety $E_G$ satisfying the
following two conditions:
\begin{enumerate}
\item $p\circ f\, =\,\rho_t\circ p$, and

\item $f$ intertwines the action of $G$ on $E_G$.
\end{enumerate}
Note that the above two conditions imply that $f$ is an algebraic isomorphism of the
principal $G$--bundle $\rho^*_t E_G$ with $E_G$.

We have the following composition on the above defined set $\mathcal G$:
$$
(t_1\, ,f_1)\cdot (t_2\, ,f_2)\, :=\,(t_1\circ t_2\, , f_1\circ f_2)\, .
$$
The inverse of $(t\, ,f)$ is $(t^{-1}\, ,f^{-1})$. These operations make
$\mathcal G$ a group. In fact, $\mathcal G$ has the structure of an affine
algebraic group defined over $\mathbb C$. Let $\mathcal A$ denote the group
of all algebraic automorphisms of the principal $G$--bundle $E_G$. So
$\mathcal A$ is a subgroup of ${\mathcal G}$ with the inclusion map
being $f\, \longmapsto\, (e\, ,f)$. We have a natural projection
$$
h\, :\, {\mathcal G}\,\longrightarrow\, T\, ,~\ ~ (t\, ,f)\,\longmapsto\, t
$$
which fits in the following exact sequence of complex affine algebraic groups:
\begin{equation}\label{e9}
0\,\longrightarrow\, {\mathcal A}\,\longrightarrow\, {\mathcal G}
\,\stackrel{h}{\longrightarrow}\, T\, .
\end{equation}
We note that there is a tautological action of ${\mathcal G}$ on $E_G$; the action of
any $(t\, ,f)\,\in\, {\mathcal G}$ on $E_G$ is
given by the map defined by $y\, \longmapsto\, f(y)$.

Now assume that $E_G$ satisfies the condition that for every $t\, \in\, T$, the pulled
back principal $G$--bundle $\rho^*_t E_G$ is isomorphic to $E_G$. This assumption
is equivalent to the statement that the homomorphism $h$ in \eqref{e9} is surjective.

In view of the above assumption, the sequence in \eqref{e9} becomes
the following short exact sequence of complex affine algebraic groups
\begin{equation}\label{e7}
0\,\longrightarrow\, {\mathcal A}\,\longrightarrow\, {\mathcal G}
\,\stackrel{h}{\longrightarrow}\, T\,\longrightarrow\, 0\, .
\end{equation}

Let ${\mathcal G}^0\, \subset\,{\mathcal G}$ be the connected component containing
the identity element. Since $T$ is connected and $h$ is surjective, the restriction
of $h$ to ${\mathcal G}^0$ is also surjective. Therefore, from \eqref{e7} we have
the following short exact sequence of affine complex algebraic groups
\begin{equation}\label{e8}
0\,\longrightarrow\, {\mathcal A}^0\,\stackrel{\iota_{\mathcal A}}{\longrightarrow}\,
{\mathcal G}^0\,\stackrel{h^0}{\longrightarrow}\, T\,\longrightarrow\, 0\, ,
\end{equation}
where ${\mathcal A}^0\,:=\, {\mathcal A}\bigcap {\mathcal G}^0$, and
$h^0\,:=\, h\vert_{{\mathcal G}^0}$.

Take a maximal torus $T_{\mathcal G}\, \subset\, {\mathcal G}^0$. From
\eqref{e8} it follows that the restriction
$$
h'\,:=\, h^0\vert_{T_{\mathcal G}}\, :\, T_{\mathcal G}
\, \longrightarrow\, T
$$
is surjective. Define $T_{\mathcal A}\,:=\, {\mathcal A}^0\bigcap T_{\mathcal G}
\,\subset\, T_{\mathcal G}$ using the homomorphism $\iota_{\mathcal A}$ in \eqref{e8}.
Therefore, from \eqref{e8} we have the following short exact sequence of
algebraic groups
\begin{equation}\label{e9p}
0\,\longrightarrow\, T_{\mathcal A}\,\stackrel{\iota_{\mathcal A}\vert_{T_{\mathcal
A}}}{\longrightarrow}\, T_{\mathcal G}\,\stackrel{h'}{\longrightarrow}\, T\,\longrightarrow\, 0\, .
\end{equation}
Recall that $\mathcal G$ has a tautological action on $E_G$. Therefore, the
subgroup $T_{\mathcal G}$ has a tautological action on $E_G$ which is the restriction
of the tautological action of $\mathcal G$.

Now we assume that the group $G$ is reductive.

A parabolic subgroup of $G$ is a connected Zariski
closed subgroup $P\, \subset\, G$ such that the variety $G/P$ is
projective. For a parabolic subgroup $P$, its unipotent radical will be denoted
by $R_u(P)$. A Levi subgroup of $P$ is a connected reductive subgroup $L(P)\,
\subset\, P$ such that the composition
$$
L(P)\,\hookrightarrow\, P\, \longrightarrow\, P/R_u(P)
$$
is an isomorphism. Levi subgroups exists, and any two Levi subgroups of $P$
differ by conjugation by an element of $R_u(P)$
\cite[p. 184--185, \S~30.2]{Hu}, \cite[p. 158, 11.22, 11.23]{Bo}.

Let $\text{Ad}(E_G)\, :=\, E_G\times^G G\,\longrightarrow\, M$ be the adjoint
bundle associated to $E_G$ for the adjoint action of $G$ on itself. The fibers
of $\text{Ad}(E_G)$ are groups identified with $G$ up to an inner automorphism;
the corresponding Lie algebra bundle is $\text{ad}(E_G)$. We note that
$\mathcal A$ in \eqref{e7} is the space of all algebraic sections of $\text{Ad}(E_G)$.

Using the action of $T_{\mathcal A}$ on $E_G$, we have
\begin{itemize}
\item{} a Levi subgroup $L(P)$ of a parabolic subgroup $P$ of $G$, and

\item an algebraic reduction of structure group $E_{L(P)}\, \subset\,
E_G$ of $E_G$ to $L(P)$ which is preserved by the tautological action
of $T_{\mathcal G}$ on $E_G$,
\end{itemize}
such that the image of $T_{\mathcal A}$ in $\text{Ad}(E_G)$ (recall that
the elements of $\mathcal A$ are sections of $\text{Ad}(E_G)$) lies in the
connected component, containing the identity element, of the
center of each fiber of $\text{Ad}(E_{L(P)})\, \subset\, \text{Ad}(E_G)$
(see \cite{BBN}, \cite{BP} for the construction of $E_{L(P)}$). The construction
of $E_{L(P)}$ requires fixing a point $z_0$ of $E_G$, and $E_{L(P)}$ contains $z_0$.
Using $z_0$, the fiber $(E_{L(P)})_{p(z_0)}$ is identified with $L(P)$.
Moreover, the evaluation, at $p(z_0)$, of the sections of
$\text{Ad}(E_G)$ corresponding to the elements of $T_{\mathcal A}$ makes
$T_{\mathcal A}$ a subgroup of the connected component,
containing the identity element, of the center of $E_{L(P)}$; in particular, this
evaluation map on $T_{\mathcal A}$ is injective (see the second paragraph in
\cite[p. 230, Section 3]{BBN}).
We briefly recall (from \cite{BBN}, \cite{BP}) the argument that
the evaluation map on semisimple elements of $\mathcal A$ is injective.
Let $\xi$ be a semisimple element of ${\mathcal A}\,=\, \Gamma(M,\, \text{Ad}(E_G))$.
Since $\xi$ is semisimple, for each point $x\, \in\, M$, the evaluation
$\xi(x)$ is a semisimple element of $\text{Ad}(E_G))_x$. The group
$\text{Ad}(E_G))_x$ is identified with $G$ up to an inner automorphism of $G$.
All conjugacy classes of semisimple element of $G$ are parametrized by
$T_G/W_{T_G}$, where $T_G$ is a maximal torus in $G$, and $W_{T_G}\,=\, N(T_G)/T_G$ is
the Weyl group with $N(T_G)$ being the normalizer of $T_G$ in $G$. We note that
$T_G/W_{T_G}$ is an affine variety. Therefore, we get a morphism $\xi'\, :\,
M\, \longrightarrow\, T_G/W_{T_G}$ that sends any $x\,\in\, M$ to the conjugacy class
of $\xi(x)$. Since $M$ is a projective variety and $T_G/W_{T_G}$ is an affine variety,
we conclude that $\xi'$ is a constant map. So if $\xi(x) =\, e$ for some
$x\,\in\, M$, then $\xi\,=\, e$ identically.

Let $Z^0_{L(P)}\, \subset\, L(P)$ be the connected component, 
containing the identity element, of the center. We note that $Z^0_{L(P)}$
is a product of copies of ${\mathbb G}_m$. Therefore, the above injective homomorphism
$T_{\mathcal A}\,\longrightarrow\, Z^0_{L(P)}$ extends to a homomorphism
$$
\eta\, :\, T_{\mathcal G}\,\longrightarrow\, Z^0_{L(P)}\, .
$$
Define
\begin{equation}\label{ep}
\eta'\, :=\, \tau\circ \eta\, ,
\end{equation}
where $\tau$ is
the inversion homomorphism of $Z^0_{L(P)}$ defined by $g\,\longmapsto\, g^{-1}$.

Consider the action of $T_{\mathcal G}$ on $E_{L(P)}$; recall that $E_{L(P)}$ is 
preserved by the tautological action of $T_{\mathcal G}$ on $E_G$. We can twist this 
action on $E_{L(P)}$ by $\eta'$ in \eqref{ep}, because the actions of $Z^0_{L(P)}$ 
and $L(P)$ on $E_{L(P)}$ commute. For this new action, the group $T_{\mathcal A}$ 
clearly acts trivially on $E_{L(P)}$.

Consider the above action of $T_{\mathcal G}$ on $E_{L(P)}$ constructed using $\eta'$.
Since $T_{\mathcal A}$ acts trivially on $E_{L(P)}$, the action of $T_{\mathcal G}$
on $E_{L(P)}$ descends to an action of $T$ on $E_{L(P)}$ (see \eqref{e9p}). The principal $G$--bundle
$E_G$ is the extension of structure group of $E_{L(P)}$ using the inclusion of $L(P)$ in $G$.
Therefore, the above action on $T$ on $E_{L(P)}$ produces an action of $T$ on $E_G$. More
precisely, the total space of $E_G$ is the quotient of $E_{L(P)}\times G$ where two elements
$(z_1\, ,g_1)$ and $(z_2\, ,g_2)$ of $E_{L(P)}\times G$ are identified if there is an
element $g\,\in\, L(P)$ such that $z_2\,=\, z_1g$ and $g_2\,=\, g^{-1}g_1$. Now the
action of $T$ on $E_{L(P)}\times G$, given by the above action of $T$ on $E_{L(P)}$ and
the trivial action of $T$ on $G$, descends to an action of $T$ on the quotient space
$E_G$. Consequently, $E_G$ admits a $T$--equivariant structure.

Therefore, we have proved the following:

\begin{proposition}\label{prop3}
Let $G$ be reductive, and let $E_G\,\longrightarrow\, M$ be a principal $G$--bundle such that
for every $t\, \in\, T$, the pulled back principal $G$--bundle $\rho^*_t E_G$ is isomorphic to $E_G$.
Then $E_G$ admits a $T$--equivariant structure.
\end{proposition}

For vector bundles on $M$, Proposition \ref{prop3} was proved by Klyachko \cite[p. 342, 
Proposition 1.2.1]{Kl}.

\subsection{Equivariance property from a logarithmic connection}

\begin{theorem}\label{thm1}
Let $G$ be reductive, and let $p\, :\, E_G\,\longrightarrow\, M$ be a principal $G$--bundle
admitting a logarithmic connection whose singularity locus is contained in the divisor
$D\,=\, M\setminus M^0$. Then $E_G$ admits a $T$--equivariant structure.
\end{theorem}

\begin{proof}
Since $E_G$ admits a logarithmic connection, by definition, there is a homomorphism
of coherent sheaves
$$
\delta\, :\, TM(-\log D)\,\longrightarrow\,\text{At}(E_G)(-\log D)
$$
such that $\phi\circ\delta$ is the identity automorphism of $TM(-\log D)$, where
$\phi$ is the homomorphism in \eqref{e5}. Let
$$
\widehat{\delta}\, :\, H^0(M, \, TM(-\log D))\,\longrightarrow\,H^0(M,\,
\text{At}(E_G)(-\log D))
$$
be the homomorphism of global sections given by $\delta$. From Lemma \ref{prop1}(2) we
know that $H^0(M, \, TM(-\log D))$ is the Lie algebra $\mathfrak t$ of $T$.

We will now show that there is a natural injective homomorphism
\begin{equation}\label{theta}
\theta\, :\, H^0(M,\, \text{At}(E_G)(-\log D))\,\longrightarrow\,\text{Lie}(\mathcal G)\, ,
\end{equation}
where $\text{Lie}(\mathcal G)$ is the Lie algebra of the group $\mathcal G$ in \eqref{e9}.

The elements of $\text{Lie}(\mathcal G)$ are all holomorphic sections $s\, \in\,
H^0(M,\, \text{At}(E_G))$ such that the vector field $\phi(s)$, where $\phi$ is the projection
in \eqref{e3}, is of the form $\beta(s')$, where $s'\, \in\, \mathfrak t$ and $\beta$ is the
homomorphism in \eqref{beta}. Now, if
$$
s\, \in\, H^0(M,\, \text{At}(E_G)(-\log D))\, \subset\, H^0(M,\, \text{At}(E_G))\, ,
$$
then $\phi(s)$ is a holomorphic section of $TM(-\log D)$ (see \eqref{e5}). From
Lemma \ref{prop1}(2) it now follows that $\phi(s)$ is of the form $\beta(s')$,
where $s'\, \in\, \mathfrak t$. This gives us the injective homomorphism in
\eqref{theta}.

Finally, consider the composition
$$
\theta\circ\widehat{\delta}\, :\, {\mathfrak t}\,=\, H^0(M, \, TM(-\log D))\,\longrightarrow\,
\text{Lie}(\mathcal G)\, .
$$
{}From its construction it follows that
$$
(dh)\circ \theta\circ\widehat{\delta}\,=\, {\rm Id}_{\mathfrak t}\, ,
$$
where $dh\, :\, \text{Lie}(\mathcal G)\,\longrightarrow\,{\mathfrak t}$ is the
homomorphism of Lie algebras given by $h$ in \eqref{e9}. In particular, $dh$ is
surjective. Since $T$ is connected, this immediately implies that the homomorphism
$h$ is surjective. Now from Proposition \ref{prop3} it follows that
$E_G$ admits a $T$--equivariant structure.
\end{proof}


\begin{thebibliography}{ZZZZ}

\bibitem[At]{At} M. F. Atiyah, Complex analytic connections in fibre
bundles, {\it Trans. Amer. Math. Soc.} {\bf 85} (1957), 181--207.

\bibitem[BBN]{BBN} V. Balaji, I. Biswas and D. S. Nagaraj,
Krull-Schmidt reduction for principal bundles, \textit{Jour. Reine Angew. Math.}
{\bf 578} (2005), 225--234.

\bibitem[BP]{BP} I. Biswas and A. J. Parameswaran, On the equivariant reduction of 
structure group of a principal bundle to a Levi subgroup, {\it Jour. Math. Pures 
Appl.} {\bf 85} (2006), 54--70.

\bibitem[Bo]{Bo} A. Borel, {\it Linear algebraic groups}, Second edition, Graduate 
Texts in Mathematics, 126, Springer-Verlag, New York, 1991.

\bibitem[Fu]{Fu} W. Fulton, {\it Introduction to Toric Varieties}, Annals of 
Mathematics Studies, 131, Princeton University Press, Princeton, 1993.

\bibitem[Hu]{Hu} J. E. Humphreys, \textit{Linear algebraic
groups,} Graduate Texts in Mathematics, Vol. 21,
Springer-Verlag, New York, Heidelberg, Berlin, 1987.

\bibitem[Kl]{Kl} A. A. Klyachko, Equivariant bundles on toral varieties, {\it Math. 
USSR Izvestiya} {\bf 35} (1990), 337--375.

\end{thebibliography}
\end{document}